\documentclass[a4paper, 10pt]{amsart}

\usepackage[utf8]{inputenc}
\usepackage{amssymb}
\usepackage{kpfonts}
\usepackage[english]{babel}
\usepackage[size=2em,nohug,midshaft,scriptlabels,PostScript=dvips]{diagrams}

\theoremstyle{plain}
\newtheorem{theorem}{Theorem}[section]

\newtheorem{proposition}[theorem]{Proposition}

\theoremstyle{definition}
\newtheorem{definition}[theorem]{Definition}
\newtheorem{example}[theorem]{Example}

\theoremstyle{remark}

\numberwithin{equation}{section}

\usepackage{accents,mathtools}
\DeclareMathSymbol{\widehatsym}{\mathord}{largesymbols}{"62}
\newcommand\lowerhathatsym{%
  \text{\smash{\rlap{\raisebox{-1.1ex}{$\widehatsym$}}\raisebox{-1.4ex}{%
    $\widehatsym$}}}}
\newcommand\hathatnospace[1]{%
  \mathchoice
    {\accentset{\displaystyle\lowerhathatsym}{#1}}
    {\accentset{\textstyle\lowerhathatsym}{#1}}
    {\accentset{\scriptstyle\lowerhathatsym}{#1}}
    {\accentset{\scriptscriptstyle\lowerhathatsym}{#1}}
}
\newcommand\widehathat[1]{%
   \mathclap{\phantom{\hat{#1}}}{\hathatnospace{#1}}}

\newcommand{\ZZ}{\mathbb{Z}} 
\newcommand{\QQ}{\mathbb{Q}} 
\newcommand{\CC}{\mathbb{C}} 
\newcommand{\iso}{\cong}      
\newcommand{\PP}{\mathbb{P}}  
\newcommand{\sheaf}[1]{\mathscr{#1}} 
\newcommand{\OO}{{\sheaf{O}}}   
\newcommand{\tsum}{{\textstyle\sum}}  
\newcommand{\res}[2]{\left.#1\right|_{#2}} 
\newcommand{\isoto}{\overset{\smash{\raisebox{-0.6ex}{$\textstyle\sim\;$}}}{\to}} 
\newcommand{\tensor}{\otimes}
\newcommand{\HHom}{\sheaf{H}\!\mathit{om}} 

\DeclareMathOperator{\rk}{rk} 
\DeclareMathOperator{\Pic}{Pic}
\DeclareMathOperator{\ch}{ch}
\DeclareMathOperator{\Hom}{Hom}
\DeclareMathOperator{\End}{End}
\DeclareMathOperator{\Ext}{Ext}
\DeclareMathOperator{\Spec}{Spec}
\DeclareMathOperator{\Sym}{Sym}
\DeclareMathOperator{\tr}{tr} 
\DeclareMathOperator{\DT}{DT} 
\DeclareMathOperator{\Hilb}{Hilb}

\begin{document}

\title[Counting sheaves]{Counting sheaves on Calabi-Yau and abelian threefolds}

\author{Martin G. Gulbrandsen}
\address{Stord/Haugesund University College, Norway}
\email{martin.gulbrandsen@hsh.no}


\date{}

\begin{abstract}
We survey the foundations for Donaldson--Thomas invariants for
stable sheaves on algebraic threefolds with trivial canonical bundle, with
emphasis on the case of abelian threefolds.
\end{abstract}

\maketitle

Let $X$ be a Calabi--Yau threefold, in the weak sense that the canonical
sheaf $\omega_X$ is trivial.
The aim of Donaldson--Thomas theory is to make sense of ``counting'' the number of stable
sheaves on $X$.

This text consists of two parts: in the first part, Section \ref{sec:count}, we
give an informal and somewhat simplified introduction to the foundations for
Donaldson--Thomas invariants, following Behrend--Fantechi \cite{BF97, BF2008},
Li--Tian \cite{LT98}, Siebert \cite{siebert2004}, Thomas \cite{thomas2000},
Huybrechts--Thomas \cite{HT2010}, Behrend \cite{behrend2009}, and
Joyce--Song \cite{JS2011}. We put some emphasis on the possibility of having
nontrivial $H^1(\OO_X)$, so that line bundles may deform, as we have the
abelian situation in mind: in the second part, Section \ref{sec:abelian}, we
discuss recent work by the author \cite{gulbrandsen2011b}, where we modify the
standard setup surveyed in the first part, to obtain nontrivial
Donaldson--Thomas invariants for abelian threefolds $X$.

\section{Virtual counts}\label{sec:count}

We work over $\CC$ for simplicity.
Fix a polarization $H$ on the Calabi--Yau threefold $X$,
and let $M$ denote the Simpson moduli space \cite{simpson94, HL97}
of $H$-stable coherent sheaves on $X$,
with fixed Chern character $\ch \in \bigoplus_p H^{2p}(X,\QQ)$. We assume that $M$ is compact,
for instance by choosing the Chern classes such that strictly semistable
sheaves are excluded for numerical reasons. For simplicity we shall also
assume that there is a universal family, denoted $\sheaf{F}$, on $X\times M$.

The \emph{virtual dimension} is the guess at $\dim M$ one obtains from
deformation theory; at any point $p \in M$ it is
\begin{equation*}
d^{\text{vir}} =
\dim \underbrace{\Ext^1(\sheaf{F}_p,\sheaf{F}_p)}_{\text{tangents}}
- \dim \underbrace{\Ext^2(\sheaf{F}_p,\sheaf{F}_p)}_{\text{obstructions}}
= 0
\end{equation*}
by Serre duality. Our aim is to ``count $M$'', even if the prediction fails,
so that $M$ has positive dimension. The number arrived at is the
Donaldson--Thomas invariant $\DT(M)$.

\subsection{Deformation invariance}\label{sec:deform}

Here is a thought model: if $M$ fails to be finite, suppose we can deform $X$ to a new
Calabi-Yau threefold $X'$ such that the corresponding moduli space $M'$ of sheaves
on $X'$ is finite. Then we want to declare that $M$ should have had
the same number of points as $M'$, but for some reason $M$ came out oversized.
So we define the virtual count $\DT(M)$ as the number of points in
$M'$. Of course we ask whether this count
is independent of the chosen deformation and, if it is, whether it can be
phrased intrinsically on $M$.
The answer is affirmative, and this intrinsically defined
invariant is the Donaldson--Thomas invariant.

This presentation is misleading in that we rarely can find a finite $M'$,
but it motivates the following demands: the virtual count should
be such that
\begin{itemize}
\item if $M$ is a finite and reduced, then $\DT(M)$ is its number of points
\item $\DT(M_t)$ is constant in (smooth) families $X_t$.
\end{itemize}
Note that the topological Euler characteristic does
specialize to the number of points when $M$ is finite and reduced, but it is certainly
not invariant under deformation of $X$. We return to Euler characteristics
in Section \ref{sec:weighted}.

\subsection{Virtual fundamental class}\label{sec:virtual}

The Donaldson--Thomas invariant is defined as the degree of the
\emph{virtual fundamental class} $[M]^\text{vir}$, which is a
Chow class on $M$ of dimension $d^{\text{vir}}=0$. It is in some sense
a characteristic class attached to an obstruction theory on $M$.
The same construction underlies Gromov--Witten invariants,
where $M$ is instead a moduli space for stable maps (whose virtual
fundamental class has positive dimension).
This machinery was first developed by Li--Tian \cite{LT98}; our presentation
follows Behrend--Fantechi \cite{BF97}, and is also influenced by Siebert \cite{siebert2004}.

Here is a toy model (cf.\ \cite[Section 3]{thomas2000}), which serves as a guide for the actual construction.
Suppose the moduli space comes out naturally as the zero locus $M = Z(s)$
of a section $s\in \Gamma(V, E)$ of a vector bundle on a smooth variety $V$.
Then the expected dimension of $M$ is
\begin{equation*}
d^{\text{vir}} = \dim V - \rk E
\end{equation*}
and if this is indeed the dimension of $M$, then its fundamental class
is $c_{\text{top}}(E)$. But in any case, there is the localized top
Chern class $\ZZ(s)$ \cite[Section 14.1]{fulton98}, which is a degree
$d^{\text{vir}}$ class in the Chow group of $M$. This should be our
$[M]^\text{vir}$.

The section $s$ embeds $M$ into $E$; recall that deformation to
the normal cone in this context \cite[Remark 5.1.1]{fulton98} says that
as $\lambda\to \infty$, the locus $\lambda s(M)\subset E$ becomes
the normal cone $C_{M/V}\subset \res{E}{M}$.
The localized top Chern class $\ZZ(s)$
is the (refined) intersection of $C_{M/V}$ with the zero section
of $\res{E}{M}$. Thus we can forget about $V$: all we need to be
able to
write down our toy virtual fundamental class $[M]^\text{vir} = \ZZ(s)$ is a cone ($C_{M/V}$) in a vector
bundle ($\res{E}{M}$) on $M$.
It turns out that the normal cone $C_{M/V}$, or at least an essential
part of it, is in some sense intrinsic to $M$, whereas an embedding
into a vector bundle is a (perfect) obstruction theory on $M$.

Now return to the actual moduli space $M$: choose an embedding $M\subset V$
into a smooth variety $V$ (our $M$ is projective, so we may take $V=\PP^n$).
Let $\sheaf{I}\subset \OO_V$ be the ideal of the embedding.
The natural map
\begin{equation}\label{eq:cot}
L_M\colon \sheaf{I}/\sheaf{I}^2 \to \res{\Omega_V}{M},
\end{equation}
considered as a complex with objects in degrees $-1$ and $0$, is the truncated
cotangent complex for $M$.
We define
\begin{align*}
T_V &= \Spec \Sym \Omega_V,\\
N_{M/V} &= \Spec \Sym \sheaf{I}/\sheaf{I}^2,\\
C_{M/V} &= \Spec \textstyle\bigoplus \sheaf{I}^d/\sheaf{I}^{d+1}
\end{align*}
(the first two are vector space fibrations, with possibly varying
fibre dimensions, and the last one is a cone fibration)
so \eqref{eq:cot} gives a map
\begin{equation}\label{eq:normal}
\res{T_V}{M} \to N_{M/V}
\end{equation}
and there is an embedding $C_{M/V}\subseteq N_{M/V}$.
Now Behrend--Fantechi define stack quotients
\begin{align*}
N_M &= [N_{M/V} / \res{T_V}{M}], &
C_M &= [C_{M/V} / \res{T_V}{M}]
\end{align*}
and prove that they are independent of $V$. They are the
\emph{intrinsic normal space} and \emph{intrinsic normal cone}.
The reader not comfortable with stacks can safely view $N_M$
as the map \eqref{eq:normal} between vector space fibrations,
modulo some equivalence relation, and similarly for $C_M$
(this viewpoint is carried further by Siebert \cite{siebert2004}).
But it is also useful for intuition to think of them as somewhat weird vector space and
cone fibrations over $M$. For instance, the fibre of $N_M$ over
a smooth point $p\in M$ is the trivial vector space together with the
stabilizer group $T_M(p)$.

So we have the intrinsically defined (stacky) normal cone $C_M$ on $M$,
embedded into $N_M$. But the latter is a (stacky) vector space fibration,
which is not necessarily locally free. So the lacking piece of data
is an embedding of $N_M$ into a vector bundle. We enlarge our notion
of vector bundles to allow stack quotients $E = [E^1/E^0]$, where
$E^0\to E^1$ is a linear map of vector bundles on $M$. Ad hoc, we
define a map $f\colon N_M\subset E$ to be something induced by a commutative
square
\begin{equation}\label{eq:N-embed}
\begin{diagram}
\res{T_M}{V} & \rTo & N_{M/V} \\
\dTo^{f^0} && \dTo^{f^1} \\
E^0 & \rTo & E^1
\end{diagram}
\end{equation}
and to count as an embedding, the map $f^1$ should
take distinct $\res{T_M}{V}$-orbits in $N_{M/V}$ to distinct
$E^0$-orbits in $E^1$, i.e.~the induced map on cokernels should be
a \emph{monomorphism}. Furthermore we require that the induced
map on kernels should be an \emph{isomorphism},
so that stabilizer groups in $N_M$ and
in $E$ agree.

In summary, we want to equip the scheme $M$ with a (stacky)
vector bundle $E$ and an embedding $N_M\subset E$ of the
intrinsic normal space. Then we define the virtual fundamental class
$[M]^{\text{vir}}$ as the intersection of the cone $C_M\subset E$
with the zero section in $E$. (This does make sense on stacks
\cite{kresch99}.)

\subsection{Obstruction theory}

Obstruction theory is a systematic answer to the problem of extending a morphism
$f\colon T \to M$ to an infinitesimal thickening $T \subset \overline{T}$, where
the ideal $\sheaf{I}\subset \OO_{\overline{T}}$ of $T$ has square zero. We may and will
assume that $T$ and $\overline{T}$ are affine. To each such situation,
there is a canonical obstruction class \cite[Section 4]{BF97}
\begin{equation*}
\omega \in \Ext^1_T(f^*L_M, \sheaf{I})
\end{equation*}
whose vanishing is equivalent to the existence of a morphism $\overline{f}\colon \overline{T}\to M$
extending $T$. Moreover, when $\omega=0$, the set of such extensions $\overline{f}$ form in a natural
way a torsor under $\Hom_T(f^*L_M, \sheaf{I})$.
These statements, if unfamiliar to the reader, may be taken on trust for the purposes
of this text.

Obstruction theory is connected with the construction of virtual fundamental
classes as follows:
diagram \eqref{eq:N-embed} is obtained by applying $\Spec \Sym (-)$ to
the dual diagram of coherent sheaves
\begin{equation*}
\begin{diagram}
\res{\Omega_M}{V} & \lTo & \sheaf{I}/\sheaf{I}^2 \\
\uTo^{\phi^0} && \uTo^{\phi^{-1}} \\
\sheaf{E}^0 & \lTo & \sheaf{E}^{-1}
\end{diagram}
\end{equation*}
i.e.~a morphism of complexes
\begin{equation*}
\phi\colon \sheaf{E} \to L_M
\end{equation*}
(to be precise, this happens in the derived category, so quasi-iso\-mor\-phisms
are inverted).

\begin{theorem}[Behrend--Fantechi \cite{BF97}]\label{thm:obstr}
Let $\sheaf{E}$ be a complex of locally free sheaves concentrated in nonpositive
degrees and let
$\phi\colon \sheaf{E} \to L_M$ be a morphism (in the derived category).
Then the following are equivalent:
\begin{itemize}
\item[(i)] For each deformation situation $T\subset \overline{T}$, $f\colon T\to M$,
there is a morphism $\overline{f}\colon \overline{T}\to M$ extending $f$ if and only if
\begin{equation*}
\phi^*(\omega) \in \Ext^1_T(f^*\sheaf{E}, \sheaf{I})
\end{equation*}
vanishes; furthermore when $\phi^*(\omega)=0$, the set of such extensions $\overline{f}$
form a torsor under $\Hom_T(f^*\sheaf{E}, \sheaf{I})$.
\item[(ii)] The morphism $\phi$ induces an isomorphism in degree $0$ and an epimorphism
in degree $-1$.
\end{itemize}
\end{theorem}

The condition that $f$ in Diagram \eqref{eq:N-embed} induces an isomorphism
on kernels and a monomorphism on cokernels translates precisely to the condition on $\phi$
in (ii) in the theorem. Thus an embedding $N_M\subset E$ is equivalent to an obstruction
theory of a particular kind:

\begin{definition}[Behrend--Fantechi \cite{BF97}]
A \emph{perfect obstruction theory} on $M$ is a two term complex of locally free sheaves
$\sheaf{E}$ together with a morphism $\phi\colon \sheaf{E}\to L_M$,
such that $\phi$ induces an isomorphism in degree $0$ and an epimorphism
in degree $-1$.
\end{definition}

The point is, of course, that our moduli space $M$, or rather the subspace $M(\sheaf{L})\subset M$
of sheaves with fixed determinant line bundle $\sheaf{L}$, carries a natural perfect obstruction theory.
To construct the obstruction theory,
we assume the rank $r$ is nonzero, and consider the trace map
\begin{equation*}
\tr\colon \HHom(\sheaf{F},\sheaf{F}) \to \OO_{X\times M}
\end{equation*}
(do this in the derived category, so $\HHom$ means derived $\HHom$;
if $\sheaf{F}$ is locally free it doesn't matter, of course).
Let $\sheaf{F}_T$ be the sheaf on $T\times X$ obtained by pulling back the universal
family $\sheaf{F}$ along $f\colon T \to M(\sheaf{L})$. To extend $f$ to $\overline{T}$ is the same
as to extend $\sheaf{F}_T$ to a $\overline{T}$-flat family on $\overline{T}\times X$ with constant determinant.
The trace map induces
\begin{equation*}
\tr^i\colon \Ext^i(\sheaf{F}_T, \sheaf{F}_T\tensor_{\OO_T} \sheaf{I}) \to H^i(\sheaf{I}).
\end{equation*}
We will use a subscript $0$ on $\HHom$ and $\Ext^i$ to indicate the kernels of
$\tr$ and $\tr^i$.
By reasonably elementary arguments (see e.g.\ Thomas \cite{thomas2000}), the existence
of such an extension is equivalent to the vanishing of a certain class
\begin{equation}\label{eq:ext2}
\omega \in \Ext^2_0(\sheaf{F}_T, \sheaf{F}_T\tensor_{\OO_T} \sheaf{I}).
\end{equation}
Moreover, when $\omega=0$, the set of extensions of $\sheaf{F}_T$ form a torsor
under
\begin{equation}\label{eq:ext1}
\Ext^1_0(\sheaf{F}_T, \sheaf{F}_T\tensor_{\OO_T} \sheaf{I}).
\end{equation}
This elementary obstruction theory can be lifted to a Behrend--Fantechi type
theory, and for this step we will be brief: 
the diagonal map $\OO_{X\times M} \to \HHom(\sheaf{F},\sheaf{F})$
composed with the trace map is multiplication by the rank $r$, hence there
is a splitting
\begin{equation}\label{eq:split}
\HHom(\sheaf{F},\sheaf{F}) = \HHom_0(\sheaf{F},\sheaf{F}) \oplus \OO_{X\times M}.
\end{equation}
There is a natural morphism (essentially the Atiyah class of $\sheaf{F}$ \cite{HT2010})
\begin{equation}\label{eq:obstr}
\phi\colon \sheaf{E} = (p_{2*}\HHom_0(\sheaf{F},\sheaf{F}))^\vee[-1] \to L_M
\end{equation}
(again, derived functors), whose restriction to $M(\sheaf{L})$ is a perfect
obstruction theory. In fact, there are $\OO_T$-linear isomorphisms
\begin{equation*}
\Ext^i(\sheaf{F}_T, \sheaf{F}_T\tensor_{\OO_T} \sheaf{I})
\iso
\Ext^{i-1}(f^*\sheaf{E}, \sheaf{I})
\end{equation*}
such that the obstruction class in \eqref{eq:ext2} agrees with the one in Theorem \ref{thm:obstr} (i),
and the torsor structures are the same \cite[Theorem 4.1]{HT2010}.

If we used the full $\HHom$ instead of the trace free $\HHom_0$ in \eqref{eq:obstr},
the complex $\sheaf{E}$ would be too big in two ways:
firstly, it would not be concentrated in degrees $[-1,0]$; secondly and more seriously,
even if we truncate it, it would contain a trivial summand by \eqref{eq:split}, causing the virtual
fundamental class to be zero (just as, in our toy model in Section \ref{sec:virtual}, the
top Chern class of a vector bundle with a trivial summand is zero).
This is why we are led to fixing the determinant, as the trace free part of $\HHom(\sheaf{F},\sheaf{F})$
is precisely what controls the deformation theory
for sheaves with fixed determinant.

\begin{definition}[Thomas \cite{thomas2000}, Huybrechts--Thomas \cite{HT2010}]
The Don\-ald\-son--Tho\-mas invariant $\DT(M(\sheaf{L}))$ of $M(\sheaf{L})$ is the degree of the virtual fundamental
class $[M(\sheaf{L})]^{\text{vir}}$ associated to the canonical perfect obstruction theory
\eqref{eq:obstr}.
\end{definition}

The Donaldson--Thomas invariant does fulfill our two requirements from Section \ref{sec:deform}.
Firstly, if $M(\sheaf{L})$ happens to be finite and reduced, and somewhat more generally: finite
and a complete intersection, then the truncated cotangent complex is itself
a perfect obstruction theory, and the associated virtual fundamental class
is the usual fundamental class, hence its degree is the length of $M(\sheaf{L})$ as a finite scheme.
Secondly, the obstruction theory we have sketched above generalizes to the
relative situation of a moduli space $M\to S$ for sheaves on the fibres of
a family $X\to S$ of Calabi--Yau threefolds. The relative obstruction theory gives rise to a virtual
fundamental class on the whole family, which restricts to the fibrewise virtual fundamental class.
Consequently, the degree of the virtual fundamental class is constant
among the fibres.

\subsection{Behrend's weighted Euler characteristic}\label{sec:weighted}

A priori, the Donaldson--Thomas invariant defined above may depend on
the choice of obstruction theory on $M(\sheaf{L})$. But in fact,
the invariant can be rephrased entirely in terms of the intrinsic
geometry of $M(\sheaf{L})$: the decisive property of the obstruction
theory \eqref{eq:obstr} is that it is not only of virtual dimension $0$,
but it is \emph{symmetric} \cite[Definition 1.10]{BF2008}.
Roughly speaking, symmetry is a refinement of the
property that
\begin{equation*}
\Ext^1(\sheaf{F}_p,\sheaf{F}_p) \quad\text{and}\quad\Ext^2(\sheaf{F}_p,\sheaf{F}_p)
\end{equation*}
are Serre dual (and so are the trace free versions).

Now, for any scheme $Y$, Behrend defines an integral invariant
\begin{equation*}
\nu\colon Y \to \ZZ
\end{equation*}
with the properties (among
others) that $\nu(p)$ only depends on an \'etale neighbourhood of $p\in Y$,
at smooth points $\nu = 1$, and $\nu^{-1}(n)$ is a constructible
subset for all $n\in \ZZ$.

\begin{theorem}[Behrend \cite{behrend2009}]
If $Y$ is a compact scheme with a perfect symmetric obstruction
theory, then the degree of the associated virtual fundamental class
equals the $\nu$-weighted Euler characteristic
\begin{equation*}
\tilde{\chi}(Y) = \sum_{n\in \ZZ} n \chi(\nu^{-1}(n)).
\end{equation*}
\end{theorem}

The theorem has at least three important consequences: Firstly, as promised, the Donaldson--Thomas invariant
is an intrinsic invariant of $M(\sheaf{L})$. Secondly, the weighted
Euler characteristic is directly accessible for computation in examples \cite{BF2008}.
But deformation invariance does not follow from the weighted Euler characteristic formulation;
this is a consequence of the virtual fundamental class machinery.
So, for instance, in the presence of strictly semi-stable sheaves,
one might attempt to define generalized Donaldson--Thomas invariants
as the weighted Euler characteristic of either the non-compact moduli space $M(\sheaf{L})$ or the compactified
moduli space $\overline{M}(\sheaf{L})$ for semi-stable sheaves,
but these numbers would not be deformation invariant.
Still, and this is the third consequence, Behrend's weighted Euler characteristic
is the starting point for Joyce--Song's \cite{JS2011} correct (i.e.\ deformation invariant) and somewhat mysterious way of
counting strictly semi-stable sheaves (see also Kontsevich--Soibelman \cite{KS2008}). In a nontrivial manner, these generalized Donaldson--Thomas
invariants take into account all ways of putting together stable (Jordan--H\"older) factors
to form semi-stable sheaves as iterated extensions.
At present, this theory only covers the situation where $H^1(\OO_X)=0$, so
that line bundles do not deform.

\section{Abelian threefolds}\label{sec:abelian}

Let $X$ be an abelian threefold. The theory outlined in the first part of this text
applies, but almost always results in vanishing Donaldson--Thomas invariants.
We will investigate why this is so, and how the setup can be adjusted to give nontrivial
invariants.

The Chern character of the sheaves parametrized by $M$ will be written
\begin{equation}\label{eq:chern}
\ch = r + c_1 + \gamma + \chi
\end{equation}
where $r$ and $\chi$ are integers, $c_1$ is a divisor class and $\gamma$ is a curve class.

\subsection{Determinants}

Let us apply Behrend's weighted Euler characteristic to see, in concrete terms,
why the Donaldson--Thomas invariant of the full moduli space $M$ is zero. Afterwards,
we shall see that restricting to $M(\sheaf{L})$ does not help.

Assume the rank $r$ is nonzero.
Let $\delta\colon M \to \Pic^{c_1}(X) \iso \widehat{X}$ be the morphism that sends
a sheaf $\sheaf{F}_p$ to its determinant line bundle $\det(\sheaf{F}_p)$. Then $\delta$
is surjective, in fact all fibres $M(\sheaf{L}) = \delta^{-1}(\sheaf{L})$ are isomorphic. This
can be seen by letting $\widehat{X} = \Pic^0(X)$ act on $M$ by twist:
\begin{equation*}
\widehat{X}\times M \to M,\quad (\xi,\sheaf{F}) \mapsto \sheaf{F}\tensor\sheaf{P}_\xi
\end{equation*}
where we write $\sheaf{P}_\xi$ for the invertible sheaf corresponding to $\xi\in\widehat{X}$.
Since
\begin{equation*}
\det(\sheaf{F}\tensor\sheaf{P}_\xi) = \det(\sheaf{F}) \tensor \sheaf{P}_{r\xi},
\end{equation*}
it follows that every orbit in $M$ surjects onto $\Pic^{c_1}(X)$, and every fibre
of $\delta$ can be moved to any other fibre by the action of some element $\xi\in\widehat{X}$.

The topological Euler characteristic of $M$ thus equals the product of the Euler characteristics
of a fibre $M(\sheaf{L})$ and the base $\widehat{X}$, but the latter has Euler characteristic zero.
Thus $\chi(M) = 0$. Via a stratification, the same argument works for Behrend's weighted
Euler characteristic: write $M = \bigcup_n M_n$ where $M_n\subset M$ is the constructible
subset $\nu^{-1}(n)$. Each $M_n$ is invariant under the $\widehat{X}$-action, so for all $n$,
$\chi(M_n)$ equals the product of the Euler characteristic of a fibre $M_n\cap M(\sheaf{L})$
and the base $\widehat{X}$, hence is zero. Thus
\begin{equation*}
\tilde{\chi}(M) = \sum_n n \chi(M_n) = 0.
\end{equation*}

This argument applies to any, not necessarily abelian, $X$. But in the abelian case,
the weighted Euler characteristic of $M(\sheaf{L})$ is usually zero, too.
We look at an example before handling the general situation.

\begin{example}\label{ex:hilbert}
Let $\Hilb^n(X)$ be the Hilbert scheme of finite subschemes $Z\subset X$ of length $n$.
By associating with $Z$ its ideal $\sheaf{I}_Z$, we view the Hilbert scheme as
a moduli space for rank $1$ sheaves. These sheaves may be deformed either
by moving $Z$ around, or by twisting with invertible sheaves in $\Pic^0(X)$,
so the full moduli space is $M = \widehat{X}\times\Hilb^n(X)$, first projection
is the determinant map $M\to \widehat{X}$, and
\begin{equation*}
M(\OO_X) = \Hilb^n(X)
\end{equation*}
is a moduli space for rank $1$ sheaves with fixed determinant $\OO_X$.
Writing $\sum Z$ for the sum under the group law on $X$, of the zero cycle
underlying $Z$, we find a second fibration:
\begin{equation}\label{eq:sum}
\Hilb^n(X) \to X, \quad Z \mapsto \tsum Z
\end{equation}
By translation with elements $x\in X$, any fibre can be moved to any other fibre,
so that by repeating the argument above, we conclude that Behrend's weighted Euler characteristic
of $\Hilb^n(X)$ is zero.
\end{example}

The ``second fibration'' \eqref{eq:sum} on the Hilbert scheme generalizes
as follows:
let $\widehat{\delta}\colon M\to X$
be the morphism that takes a sheaf $\sheaf{F}$ to the determinant 
of its Fourier--Mukai transform $\widehat{\sheaf{F}}$ (we should warn the
reader that in the literature,
the notation $\widehat{\sheaf{F}}$ is usually reserved for WIT-sheaves \cite[Definition 2.3]{mukai81}; our $\sheaf{F}$ may well be a complex).
The invertible sheaf $\det(\widehat{\sheaf{F}})$ will be called the \emph{codeterminant} of $\sheaf{F}$.
It belongs to some component of $\Pic(\widehat{X})$,
which we identify with $\Pic^0(\widehat{X}) = X$. In general, there is no relation between
the determinant and the codeterminant. More precisely, let $X\times\widehat{X}$
act on $M$ by translation and twist:
\begin{equation*}
(X\times\widehat{X})\times M \to M,\quad (x,\xi;\sheaf{F}) \mapsto T_{-x}^*\sheaf{F}\tensor\sheaf{P}_\xi
\end{equation*}
Write $\phi_{c_1}\colon X\to \widehat{X}$
for the homomorphism $x\mapsto \OO_X(T_x^*D - D)$, for any divisor $D$ representing $c_1$.
Via Poincar\'e duality, the curve class $\gamma$ in \eqref{eq:chern}
corresponds to a divisor class on $\widehat{X}$ \cite[Proposition 1.17]{mukai87}; we shall write $\psi_\gamma\colon \widehat{X} \to \widehathat{X}=X$ for the associated homomorphism.
With this notation, the action of $X\times\widehat{X}$ on a fixed sheaf $\sheaf{F}\in M$, composed
with the determinant\slash codeterminant $(\widehat{\delta},\delta)\colon M\to X\times\widehat{X}$ is
easily computed \cite[Proposition 2.2]{gulbrandsen2011b}: it is
\begin{equation}\label{eq:isogeny}
\begin{pmatrix}
\chi & -\psi_{\gamma} \\
-\phi_{c_1} & r
\end{pmatrix} \in \End(X\times\widehat{X}).
\end{equation}
Thus the condition that the rank $r$ is nonzero, that we used to ensure that $\delta$
was a fibration, is now replaced by the condition that \emph{this matrix is an isogeny}.
Then $X\times\widehat{X}$-orbits
in $M$ surjects onto $X\times\widehat{X}$ via $(\widehat{\delta},\delta)$, all fibres are
isomorphic and the weighted Euler characteristic is zero. Fixing just one determinant does not
help: we need to fix both to obtain a nontrivial invariant.

Our object is thus the fibre $M(\sheaf{L}',\sheaf{L}) \subset M$ of $(\widehat{\delta},\delta)$,
parametrizing sheaves with determinant $\sheaf{L}$ and codeterminant $\sheaf{L}'$.
Examples indicate that its weighted Euler characteristic is nonzero in general, so
there are no further fibrations, which is a good thing, since we have run out of group
actions on $M$.

\begin{proposition}
Suppose $X$ has Picard number $1$. Then the matrix \eqref{eq:isogeny} is an isogeny
if and only if
\begin{equation*}
3r\chi \ne c_1 \gamma.
\end{equation*}
\end{proposition}

See \cite[Lemma 2.3]{gulbrandsen2011b} for a statement without the Picard number
restriction, and proof.

The proposition shows that our isogeny condition is satisfied for almost all choices of Chern classes.
If the condition does fail, as it does for instance for rank $2$ vector bundles $\sheaf{F}$
with $c_1(\sheaf{F}) = 0$, we may replace $\sheaf{F}$ with $\sheaf{F}(H)$ and try again.
One can show \cite[Proposition 3.5]{gulbrandsen2011b} that by such tricks (to be made precise in \ref{sec:K}), the inequality $3r\chi \ne c_1\gamma$ can always
be forced to hold, with the sole exception of Mukai's semi-homogeneous sheaves,
whose moduli spaces are fully understood anyway \cite{mukai78}.

This observation points to an arbitrariness in the definition of the
determinant\slash codeterminant map, which changes when $\sheaf{F}$ is
replaced by $\sheaf{F}(H)$, since the
Fourier--Mukai transform does not preserve tensor product.
We will fix this arbitrariness in Section \ref{sec:K}.
This can be contrasted with the situation for abelian surfaces \cite{yoshioka2001},
where the determinant/codeterminant pair is just the Albanese map of the moduli space $M$,
and hence is entirely intrinsic. For abelian threefolds, the moduli space $M$ is not,
in general, fibred over its Albanese:

\begin{example}
Let $C\subset X$ be a non hyperelliptic genus $3$ curve, embedded into its Jacobian
by an Abel-Jacobi map. Any deformation of $C$ is a translate $T_x(C)$ by some point $x\in X$,
and in fact the Hilbert scheme component containing $C$ is isomorphic to $X$ \cite{LS2004}.
Now consider the Hilbert scheme component $H$ parametrizing translations of $C$
together with a possibly embedded point. As in Example \ref{ex:hilbert},
$H$ can be viewed as a moduli space for rank $1$ sheaves with fixed determinant. There is a map
\begin{equation*}
H \to X^2
\end{equation*}
sending a point $T_x(C) \cup y$ in $H$ to the pair $(x,y)$. This is clearly the Albanese
map, and it is generically bijective. However, the fibre over a pair $(x,y)$ for which
$y \in T_x(C)$, consists of all embedded points in $T_x(C)$ supported at $y$, hence
is a $\PP^1$. In particular, the Albanese fibres of $H$ are not isomorphic.
\end{example}

\subsection{Translation and twist}\label{sec:K}

We assume that the matrix \eqref{eq:isogeny} is an isogeny throughout this section.

The fibres $M(\sheaf{L}',\sheaf{L})$ of $(\widehat{\delta},\delta)$ 
intersect each $X\times\widehat{X}$-orbit in finitely many points. The canonical
object lurking here is the quotient space
\begin{equation*}
K = M/X\times\widehat{X} = M(\sheaf{L}',\sheaf{L})/G
\end{equation*}
(which we consider as a Deligne--Mumford stack), where $G$ is the (finite) kernel
of the isogeny \eqref{eq:isogeny}. The point is that the ``tricks'' we alluded to above,
such as twisting with a divisor, preserves the $X\times\widehat{X}$-action on $M$,
so that $K$ does not change. More generally, suppose $Y$ is a second abelian threefold
and there exists a derived equivalence $F\colon D(X) \isoto D(Y)$. Then we may equally well
consider $M$ as a moduli space for sheaves $\sheaf{F}$ on $X$ or as a moduli space
for $F(\sheaf{F})$ on $Y$ (this may be a complex, and not a sheaf,
and for this reason we work with complexes from the start in \cite{gulbrandsen2011b}).
Orlov shows that there is an induced isomorphism $X\times\widehat{X} \iso Y\times \widehat{Y}$
such that the two actions on $M$ are compatible \cite[Corollary 2.13]{orlov2002},
so $M/X\times\widehat{X} \iso M/Y\times \widehat{Y}$. In this sense, the space $K$ is invariant
under derived equivalence, although $M(\sheaf{L}',\sheaf{L})$ is not.

Although weighted Euler characteristics of $M(\sheaf{L}',\sheaf{L})$
and $K$ make sense, it is not clear that they are of interest (in particular, whether they are
invariant under deformation of $X$) unless there is an underlying perfect
obstruction theory.

\begin{theorem}\label{thm:obstr-ab}
There is a perfect symmetric obstruction theory on $M(\sheaf{L}',\sheaf{L})$.
\end{theorem}

\begin{proof}
We sketch the main points; the details can be found in \cite{gulbrandsen2011b}.
Consider the problem of extending $f\colon T\to M(\sheaf{L}',\sheaf{L})$
over $T\subset \overline{T}$.
The trace map on $\widehat{X}\times M$,
\begin{equation*}
\widehat{\tr}\colon \HHom(\widehat{\sheaf{F}},\widehat{\sheaf{F}}) \to \OO_{\widehat{X}\times M}
\end{equation*}
together with the Fourier--Mukai-induced isomorphism
\begin{equation}\label{eq:hom-iso}
p_{2*}\HHom(\sheaf{F},\sheaf{F}) \iso p_{2*}\HHom(\widehat{\sheaf{F}},\widehat{\sheaf{F}})
\end{equation}
(in the derived category) give new trace maps
\begin{equation*}
\widehat{\tr}^i\colon \Ext^i(\sheaf{F}_T, \sheaf{F}_T\tensor_{\OO_T}\sheaf{I}) \to H^i(\sheaf{I}).
\end{equation*}
Switching back and forth between $X$ and $\widehat{X}$ we thus see that the obstruction
class $\omega$ for extending $f$ to $\overline{T}$ is a class in $\ker(\tr^2)\cap \ker(\widehat{\tr}^2)$,
and when $\omega=0$, the set of such extensions, with
fixed determinant and codeterminant, is a torsor under $\ker(\tr^1)\cap \ker(\widehat{\tr}^1)$.

Again the setup can be lifted to a Behrend--Fantechi obstruction theory: the
two trace maps taken together and pushed down to $M$
\begin{equation}\label{eq:double-trace}
p_{2*}\HHom(\sheaf{F},\sheaf{F}) \to p_{2*}\OO_{\widehat{X}\times M} \oplus p_{2*}\OO_{X\times M}
\end{equation}
(derived functors)
is a split epimorphism in degrees $1$ and $2$: in fact, the two diagonal maps
\begin{gather*}
\OO_{X\times M} \to \HHom(\sheaf{F},\sheaf{F})\\
\OO_{\widehat{X}\times M} \to \HHom(\widehat{\sheaf{F}},\widehat{\sheaf{F}})
\end{gather*}
give, via \eqref{eq:hom-iso}, a map
\begin{equation}\label{eq:double-incl}
p_{2*}\OO_{\widehat{X}\times M} \oplus p_{2*}\OO_{X\times M}
\to p_{2*}\HHom(\sheaf{F},\sheaf{F}).
\end{equation}
The composition of \eqref{eq:double-incl} with \eqref{eq:double-trace}
realizes the splitting: in degree $1$, it is an endomorphism of
$H^1(\OO_{\widehat{X}})\oplus H^1(\OO_X)$, which is nothing
but the derivative of the isogeny \eqref{eq:isogeny} at $(0,0)\in X\times\widehat{X}$, hence an isomorphism.
By duality, the degree $2$ part is an isomorphism, too. Thus, writing $\tau^{[1,2]}$ for
the truncation of a complex to degrees $[1,2]$, we have produced a splitting
\begin{equation}\label{eq:split2}
\tau^{[1,2]}p_{2*} \HHom(\sheaf{F},\sheaf{F})
= \sheaf{E} \oplus
\tau^{[1,2]}\big(p_{2*} \OO_{\widehat{X}\times M} \oplus p_{2*} \OO_{X\times M}\big).
\end{equation}
The morphism \eqref{eq:obstr} induces
\begin{equation*}
\phi\colon \sheaf{E}^\vee[-1] \to L_M
\end{equation*}
whose restriction to $M(\sheaf{L}',\sheaf{L})$ is a perfect symmetric (via duality)
obstruction theory.
\end{proof}

The two trivial summands in \eqref{eq:split2} shows, in terms of the virtual fundamental
class machinery, why it is not enough to fix one determinant, as this kills just one
of the summands.

Instead of worrying about whether the obstruction theory on $M(\sheaf{L}',\sheaf{L})$
descends to $K$, we define the Donaldson--Thomas invariant directly:

\begin{definition}
Assume \eqref{eq:isogeny} is an isogeny, and let $G$ be its (finite) kernel.
Then the Donaldson--Thomas invariant of $K$ is
\begin{equation*}
\DT(K) = \frac{1}{|G|} \deg [M(\sheaf{L}',\sheaf{L})]^\text{vir}.
\end{equation*}
\end{definition}

Theorem \ref{thm:obstr-ab} generalizes to the relative situation,
so that deformation invariance
for the virtual fundamental class of $M(\sheaf{L}',\sheaf{L})$ holds.
Since the kernel $G$ of the isogeny \eqref{eq:isogeny} has constant
order in families, the Donaldson--Thomas invariant $\DT(K)$ is invariant
under deformations of $X$. Moreover, it agrees with Behrend's weighted Euler
characteristic, hence is an intrinsic invariant of $K$.

\begin{example}
We return to the Hilbert scheme of points in Example \ref{ex:hilbert}.
The summation map $\Hilb^n(X) \to X$ agrees, up to sign, with the codeterminant
map (use that, modulo short exact sequences, $\OO_Z$ is equivalent to a sum
of skyscrapers $k(z)$, with $z\in Z$ repeated according to multiplicity,
and $\widehat{k(z)} = \sheaf{P}_z$).

The moduli space $M(\OO_{\widehat{X}},\OO_X)$ for ideals with
trivial determinant and codeterminant, is thus nothing but the locus
\begin{equation*}
K^n(X) = \left\{ Z\in \Hilb^n(X) \ \vline\ \tsum Z = 0\right\}.
\end{equation*}
(For abelian surfaces $X$, this locus $K^n(X)$ is the generalized Kummer variety
of Beauville.) The kernel $G$ of the isogeny \eqref{eq:isogeny} is in this
case the group of $n$-torsion points $X_n \subset X$ in $X\times \widehat{X}$, so
\begin{equation*}
K = K^n(X) / X_n.
\end{equation*}

Behrend--Fantechi \cite{BF2008} found that Behrend's weighted Euler characteristic of the Hilbert
scheme of $n$ points on any threefold agrees, up to sign, with the usual Euler characteristic.
Their argument can be adapted to $K^n(X)$, showing that its weighted Euler characteristic
is $(-1)^{n+1} \chi(K^n(X))$, and so
\begin{equation*}
DT(K) = \frac{1}{|X_n|} \tilde\chi(K^n(X)) = \frac{(-1)^{n+1}}{n^{6}} \chi(K^n(X)).
\end{equation*}
See \cite[Section 4.2]{gulbrandsen2011b} for a conjectural explicit formula for the Euler characteristic of $K^n(X)$.
\end{example}

\bibliographystyle{amsplain}
\bibliography{all}

\providecommand{\bysame}{\leavevmode\hbox to3em{\hrulefill}\thinspace}
\providecommand{\MR}{\relax\ifhmode\unskip\space\fi MR }
\providecommand{\MRhref}[2]{%
  \href{http://www.ams.org/mathscinet-getitem?mr=#1}{#2}
}
\providecommand{\href}[2]{#2}
\begin{thebibliography}{10}

\bibitem{behrend2009}
K.~Behrend, \emph{Donaldson-{T}homas type invariants via microlocal geometry},
  Ann. of Math. (2) \textbf{170} (2009), no.~3, 1307--1338.

\bibitem{BF97}
K.~Behrend and B.~Fantechi, \emph{The intrinsic normal cone}, Invent. Math.
  \textbf{128} (1997), no.~1, 45--88.

\bibitem{BF2008}
\bysame, \emph{Symmetric obstruction theories and {H}ilbert schemes of points
  on threefolds}, Algebra Number Theory \textbf{2} (2008), no.~3, 313--345.

\bibitem{fulton98}
W.~Fulton, \emph{Intersection theory}, second ed., Ergebnisse der Mathematik
  und ihrer Grenzgebiete. 3. Folge. A Series of Modern Surveys in Mathematics
  [Results in Mathematics and Related Areas. 3rd Series. A Series of Modern
  Surveys in Mathematics], vol.~2, Springer-Verlag, Berlin, 1998.

\bibitem{gulbrandsen2011b}
M.~G. Gulbrandsen, \emph{Donaldson--{T}homas invariants for complexes on
  abelian threefolds}, 2011, arXiv:1111.0874v1 [math.AG].

\bibitem{HL97}
D.~Huybrechts and M.~Lehn, \emph{The geometry of moduli spaces of sheaves},
  Aspects of Mathematics, E31, Friedr. Vieweg \& Sohn, Braunschweig, 1997.

\bibitem{HT2010}
D.~Huybrechts and R.~P. Thomas, \emph{Deformation-obstruction theory for
  complexes via {A}tiyah and {K}odaira-{S}pencer classes}, Math. Ann.
  \textbf{346} (2010), no.~3, 545--569.

\bibitem{JS2011}
D.~Joyce and Y~Song, \emph{A theory of generalized {D}onaldson--{T}homas
  invariants}, Mem. Amer. Math. Soc. (2011), posted on July 18, 2011, PII S
  0065-9266(2011)00630-1 (to appear in print).

\bibitem{KS2008}
M.~Kontsevich and Y.~Soibelman, \emph{Stability structures, motivic
  {D}onaldson-{T}homas invariants and cluster transformations}, 2008,
  arXiv:0811.2435v1 [math.AG].

\bibitem{kresch99}
A.~Kresch, \emph{Cycle groups for {A}rtin stacks}, Invent. Math. \textbf{138}
  (1999), no.~3, 495--536.

\bibitem{LS2004}
H.~Lange and E.~Sernesi, \emph{On the {H}ilbert scheme of a {P}rym variety},
  Ann. Mat. Pura Appl. (4) \textbf{183} (2004), no.~3, 375--386.

\bibitem{LT98}
J.~Li and G.~Tian, \emph{Virtual moduli cycles and {G}romov-{W}itten invariants
  of algebraic varieties}, J. Amer. Math. Soc. \textbf{11} (1998), no.~1,
  119--174.

\bibitem{mukai78}
S.~Mukai, \emph{Semi-homogeneous vector bundles on an {A}belian variety}, J.
  Math. Kyoto Univ. \textbf{18} (1978), no.~2, 239--272.

\bibitem{mukai81}
\bysame, \emph{Duality between {$D(X)$} and {$D(\hat X)$} with its application
  to {P}icard sheaves}, Nagoya Math. J. \textbf{81} (1981), 153--175.

\bibitem{mukai87}
\bysame, \emph{Fourier functor and its application to the moduli of bundles on
  an abelian variety}, Algebraic geometry, Sendai, 1985, Adv. Stud. Pure Math.,
  vol.~10, North-Holland, Amsterdam, 1987, pp.~515--550.

\bibitem{orlov2002}
D.~O. Orlov, \emph{Derived categories of coherent sheaves on abelian varieties
  and equivalences between them}, Izv. Ross. Akad. Nauk Ser. Mat. \textbf{66}
  (2002), no.~3, 131--158.

\bibitem{siebert2004}
B.~Siebert, \emph{Virtual fundamental classes, global normal cones and
  {F}ulton's canonical classes}, Frobenius manifolds, Aspects Math., E36,
  Vieweg, Wiesbaden, 2004, (corrected version: arXiv:math/0509076v1 [math.AG]),
  pp.~341--358.

\bibitem{simpson94}
C.~T. Simpson, \emph{Moduli of representations of the fundamental group of a
  smooth projective variety. {I}}, Inst. Hautes \'Etudes Sci. Publ. Math.
  (1994), no.~79, 47--129.

\bibitem{thomas2000}
R.~P. Thomas, \emph{A holomorphic {C}asson invariant for {C}alabi-{Y}au
  3-folds, and bundles on {$K3$} fibrations}, J. Differential Geom. \textbf{54}
  (2000), no.~2, 367--438.

\bibitem{yoshioka2001}
K.~Yoshioka, \emph{Moduli spaces of stable sheaves on abelian surfaces}, Math.
  Ann. \textbf{321} (2001), no.~4, 817--884.

\end{thebibliography}

\end{document}